\documentclass[12pt]{amsart}
\usepackage[utf8]{inputenc}
\usepackage[OT4]{fontenc}
\usepackage{amsmath}
\usepackage{amsfonts}
\usepackage{amssymb}
\usepackage{amsthm}
\usepackage{bbm}
\usepackage{hyperref}
\renewcommand{\sqrt}[1]{\left( #1 \right)^\frac12}

\newcommand{\e}[0]{\mathrm{e}}
\newcommand{\E}[0]{\mathbb{E}}
\newcommand{\T}[0]{\mathbb{T}}
\newcommand{\N}[0]{\mathbb{N}}

\newcommand{\R}[0]{\mathbb{R}}
\newcommand{\Z}[0]{\mathbb{Z}}

\renewcommand{\d}[0]{\mathrm{d}}

\newcommand{\p}[1]{\left(#1\right)}

\newcommand{\beq}{\begin{equation}}
\newcommand{\eeq}{\end{equation}}
\newcommand{\supp}[0]{\mathrm{supp}\,}

\author{Maciej Rzeszut}

\title{Fefferman multiplier theorem for Hardy martingales}
\newtheorem{thm}{Theorem}

\newtheorem{cor}[thm]{Corollary}
\newtheorem{prop}[thm]{Proposition}

\numberwithin{thm}{section}
\numberwithin{equation}{section}

\allowdisplaybreaks
\begin{document}
\begin{abstract}
A well-known theorem due to Fefferman provides a characterization of Fourier multipliers from $H^1(\mathbb{T})$ to $\ell^1$, i.e. sequences $\left(\lambda_n\right)_{n=0}^\infty$ such that 
\[\sum_{n=0}^\infty \left|\lambda_n \widehat{f}(n)\right|\lesssim \|f\|_{L^1(\mathbb{T})},\]
where $f(x)=\sum_{n=0}^\infty \widehat{f}(n)e^{inx}$. We extend it to the space $H^1\left(\mathbb{T}^\mathbb{N}\right)$ of Hardy martingales, i.e. the subspace of $L^1$ on the countable product $\mathbb{T}^\mathbb{N}$ consisting of all $f$ such that the differences $\Delta_nf=f_{n}-f_{n-1}$ of the martingale wrt the standard filtration generated by $f$ satisfy
\[\left(t\mapsto \Delta_n f\left(x_1,\ldots,x_{n-1},t\right)\right)\in H^1(\mathbb{T}). \]
The key ingredient is a theorem due to P. F. X. M\"uller stating that the classical Davis-Garsia decomposition 
\[\mathbb{E} \left(\sum_{n=0}^\infty \left|\Delta_n f\right|^2\right)^\frac{1}{2}\simeq \inf_{f=g+h} \mathbb{E}\sum_{n=0}^\infty \left|\Delta_n g\right|+ \mathbb{E}\left(\sum_{n=0}^\infty \mathbb{E}\left(\left|\Delta_n f\right|^2\mid \mathcal{F}_{n-1}\right)\right)^\frac{1}{2}\]
may be done within the space of Hardy martingales.  
\end{abstract}
\maketitle
\section{Introduction}
Suppose $X$ is a shift-invariant Banach space of functions on a compact abelian group $\mathbf{G}$. If $X\subset L^1\p{\mathbf{G}}$, then the Fourier transform is well defined on $X$ and we may ask which sequences $\lambda: \widehat{\mathbf{G}}\to \R_+$ satisfy the inequality
\beq \sum_{\gamma\in \widehat{\mathbf{G}}} \lambda_\gamma \left|\widehat{f}\p{\gamma}\right|\lesssim_\lambda \|f\|_X\eeq
for $f\in X$. They are called $X\to \ell^1$ Fourier multipliers. A complete characterization is known for $\mathbf{G}=\T$, $X=H^1\p{\T}$ due to Fefferman \cite{szw}: a sequence $\lambda:\N:=\Z_+\to \R_+$ is an $H^1\p{\T}\to\ell^1$ multiplier iff
\beq \left\|\lambda\right\|_{F}:=\sup_{a\geq 1} \sum_{k=1}^\infty \p{\sum_{j=ak}^{a(k+1)-1}\lambda_j}^2.\eeq
We are going to find analogous conditions in 2 new cases:
\begin{itemize}\item $\mathbf{G}=G^\N$ and $X=H^1\left[\p{\mathcal{F}_n}_{n=0}^\infty\right]$ where $G$ is a compact abelian group and $\p{\mathcal{F}_n}_{n=0}^\infty$ is the canonical filtration on $G^\N$;
\item $\mathbf{G}=\T^\N$ and $X=H^1_{\mathrm{last}}\p{\T^\N}$ is the subspace of $L^1\p{\T^\N}$ consisting of functions $f$ generating a martingale such that $\Delta_k f$ is an $H^1\p{\T}$ function in the $k$-th variable.
\end{itemize}
We are going to use a simple observations expressing the desired property in terms of the space dual to $X$. 
\begin{prop}\label{dualization}Let $X$ be a shift-invariant space of $\ell^2(S)$-valued functions on $\mathbf{G}$ such that $X\subset L^1\p{\mathbf{G},\ell^2(S)}$. A sequence $\lambda:\widehat{\mathbf{G}}\times S\to \R_+$ satisfies 
\beq \sum_{\gamma\in\widehat{\mathbf{G}},s\in S} \lambda_{\gamma,s}\left| \left\langle \widehat{f}\p{\gamma},e_s\right\rangle\right|\lesssim_\lambda \|f\|_X\eeq
for any $f\in X$ if and only if
\beq \sup_{\left|c_{\gamma,s}\right|=1} \left\|\sum_{\gamma,s}c_{\gamma,s}\lambda_{\gamma,s}\gamma \otimes e_s\right\|_{X^*}\lesssim 1.\eeq\end{prop}
\begin{proof}
We have
\begin{align}\sup_{\|f\|_X=1}  \sum_{\gamma\in\widehat{\mathbf{G}},s\in S} \lambda_{\gamma,s}\left| \left\langle \widehat{f}\p{\gamma},e_s\right\rangle\right| & = \sup_{\|f\|_X=1}\sup_{\left|c_{\gamma,s}\right|=1}  \sum_{\gamma,s} \lambda_{\gamma,s}c_{\gamma,s}\left\langle \widehat{f}\p{\gamma},e_s\right\rangle\\ 
&= \sup_{\|f\|_X=1}\sup_{\left|c_{\gamma,s}\right|=1}  \left\langle f,\sum_{\gamma,s} \lambda_{\gamma,s}c_{\gamma,s}\gamma\otimes e_s\right\rangle\\
&= \sup_{\left|c_{\gamma,s}\right|=1} \left\|\sum_{\gamma,s}c_{\gamma,s}\lambda_{\gamma,s}\gamma \otimes e_s\right\|_{X^*}.
\end{align}
\end{proof}
\section{Martingale Hardy spaces}
First, we are going to consider spaces of adapted sequences. Let $G$ ba a compact abelian group, $\Gamma$ be its dual, $\mathcal{F}_k$ be the sigma-algebra on $G^\N$ generated by the coordinate projection $x\mapsto \p{x_j}_{j=1}^k$ and $\mathcal{H}=\ell^2(S)$ be a Hilbert space. We define 
\beq \label{eq:adapdef}L^1\p{G^\N,\left[\p{\mathcal{F}_k}_{k=0}^\infty\right],\ell^2\p{\N,\mathcal{H}}}= \left\{f\in L^1\p{G^\N,\ell^2\p{\N,\mathcal{H}}}: f_k\text{ is }\mathcal{F}_k\text{-measurable}\right\}.\eeq
\begin{thm}\label{feffadap}The norm of a positive sequence $\lambda^{(k)}_{\gamma,s}$ where $\gamma\in \Gamma^{k}$ as a Fourier multiplier from the space $L^1\p{G^\N,\left[\p{\mathcal{F}_k}_{k=0}^\infty\right],\ell^2\p{\N,\mathcal{H}}}$ to $\ell^1\p{\bigsqcup_k \Gamma^k \times S}$ is equivalent to
\beq \label{eq:oogachaka}\sup_k \p{\sum_{j\geq k} \sum_{s\in S}\sum_{\gamma'\in \Gamma^{[k+1,j]}}\left(\sum_{\gamma\in \Gamma^k} \lambda^{(j)}_{\gamma\otimes\gamma',s}\right)^2}^\frac12.\eeq\end{thm}
\begin{proof}
We will use Proposition \ref{dualization} in conjuction with a formula for a dual norm to \eqref{eq:adapdef} (cf. \cite{weisz}). Namely, if $\varphi_k$ is a $\mathcal{F}_k$-measurable $\mathcal{H}$-valued function, then
\beq \left\|\varphi\right\|_{L^1\p{G^\N,\left[\p{\mathcal{F}_k}_{k=0}^\infty\right],\ell^2\p{\N,\mathcal{H}}}^*}\simeq \sup_k \left\|\E_k \sum_{j\geq k} \left\|\varphi_j\right\|_{\mathcal{H}}^2\right\|_{L^\infty}^\frac12 .\eeq
Thus, we calculate. 
\begin{align}
&\left\|\lambda\right\|_{L^1\p{G^\N,\left[\p{\mathcal{F}_k}_{k=0}^\infty\right],\ell^2\p{\N,\mathcal{H}}}\to\ell^1}^2\\
=& \sup_{\left|c^{(k)}_{\gamma,s}\right|=1}\left\|\sum_{k=0}^\infty\sum_{\gamma\in \Gamma^k}\sum_{s\in S}c^{(k)}_{\gamma,s}\lambda^{(k)}_{\gamma,s}\gamma\otimes e_k\otimes e_s\right\|_{L^1\p{G^\N,\left[\p{\mathcal{F}_k}_{k=0}^\infty\right],\ell^2\p{\N,\mathcal{H}}}^*}^2\\
\simeq &\sup_k \sup_{c} \sup \E_k \sum_{j\geq k}\sum_{s\in S} \left|\sum_{\gamma\in \Gamma^j} c^{(j)}_{\gamma,s}\lambda^{(j)}_{\gamma,s}\gamma\right|^2\\
=\label{eq:splitgamma} &\sup_k \sup_{c} \sup \E_k \sum_{j\geq k}\sum_{s\in S}\left|\sum_{\gamma\in \Gamma^k,\gamma'\in \Gamma^{[k+1,j]}} c^{(j)}_{\gamma\otimes\gamma',s}\lambda^{(j)}_{\gamma\otimes\gamma',s}\gamma\otimes \gamma'\right|^2\\
= \label{eq:orthon}&\sup_k \sup_{x\in G^k} \sup_{c} \sum_{j\geq k}\sum_{s\in S} \sum_{\gamma'\in \Gamma^{[k+1,j]}}\left|\sum_{\gamma\in \Gamma^k} c^{(j)}_{\gamma\otimes\gamma',s}\lambda^{(j)}_{\gamma\otimes\gamma',s}\gamma(x)\right|^2\\
= \label{eq:supswap}&\sup_k \sum_{j\geq k}\sum_{s\in S} \sum_{\gamma'\in \Gamma^{[k+1,j]}}\left(\sum_{\gamma\in \Gamma^k} \lambda^{(j)}_{\gamma\otimes\gamma',s}\right)^2.
\end{align}
Here, in \eqref{eq:splitgamma} we represented every $\gamma\in \Gamma^j$ as $\gamma\otimes \gamma'$ where $\gamma\in \Gamma^k$ and $\gamma'\in\Gamma^{[k+1,j]}$. In \eqref{eq:orthon} we used the fact that for a given $x\in G^k$, the functions $\gamma'\in \Gamma^{[k+1,j]}$ on $G^{[k+1,j]}$ are orthonormal. The equation \eqref{eq:supswap} is due to the fact that the upper bound $\left|c^{(j)}_{\gamma\otimes\gamma'}\gamma(x)\right|\leq 1$ can be attained by taking (at any given $k$, $x\in G^k$) $c^{(j)}_{\gamma\otimes\gamma',s}= \overline{\gamma(x)}$. 
\end{proof}
Because of an inequality due to Lepingle \cite{Lep}, martingale difference sequences are complemented in $L^1\p{G^\N,\left[\p{\mathcal{F}_k}_{k=0}^\infty\right],\ell^2\p{\N,\mathcal{H}}}$. Therefore, we can treat $H^1\p{G^\N,\left[\p{\mathcal{F}_k}_{k=0}^\infty\right],\mathcal{H}}$ as a complemented subspace of $L^1\p{G^\N,\left[\p{\mathcal{F}_k}_{k=0}^\infty\right],\ell^2\p{\N,\mathcal{H}}}$ by
\beq H^1\p{G^\N,\left[\p{\mathcal{F}_k}_{k=0}^\infty\right],\mathcal{H}}\ni f\mapsto \p{\Delta_k f}_{k=0}^\infty\in L^1\p{G^\N,\left[\p{\mathcal{F}_k}_{k=0}^\infty\right],\ell^2\p{\N,\mathcal{H}}}.\eeq
From this and Theorem \ref{feffadap} we immediately get
\begin{cor}
The norm of a positive sequence $\p{\lambda_{\gamma,s}}_{\gamma\in \Gamma^{\oplus\N},s\in S}$ as a Fourier multiplier from the space $H^1\p{G^\N,\left[\p{\mathcal{F}_k}_{k=0}^\infty\right],\mathcal{H}}$ to $\ell^1\p{\Gamma^{\oplus\N}\times S}$ is equivalent to 
\beq \sup_k \p{ \sum_{\gamma'\in \Gamma^{[k+1,\infty)}\setminus \{0\} }\sum_{s\in S}\left(\sum_{\gamma\in \Gamma^k} \lambda_{\gamma\otimes\gamma',s}\right)^2}^\frac12 + \sup_k \p{\sum_{s\in S} \p{\sum_{\gamma\in\Gamma^{k},\gamma_k\neq 0}\lambda_{\gamma,s}}^2}^\frac12.\eeq
\end{cor}
\begin{proof} We apply the formula \eqref{eq:oogachaka} to $\lambda^{(j)}_\gamma=\lambda_\gamma$ for $j=\max\left\{i:\gamma_i\neq 0\right\}$ and $\lambda^{(j)}_\gamma=0$ otherwise. The first summand is produced by the $j>k$ part of the sum and the second one by $j=k$.  \end{proof}
It is worth norting that if $G$ is a finite group of bounded cardinality (equivalently, the underlying filtration is regular), the second summand can be omitted, because expressions for the dual norm with $\sum_{j\geq k}$ and $\sum_{j>k}$ are equivalent. 
\section{Hardy martingales}
We are going to consider a special subspace of $L^1\p{\T^\N}$, on which the norm happens to be equivalent to the $H^1\p{\T^\N,\left[\p{\mathcal{F}_k}_{k=0}^\infty\right]}$ norm, namely the space of Hardy martingales
\beq H^1_{\mathrm{last}}\p{\T^\N}=\overline{\mathrm{span}}\bigcup_{k=1}^\infty\left\{\e^{2\pi i \langle n,x\rangle}: n=\p{n_1,\ldots,n_k,0,\ldots}\text{ and }n_k>0\right\}\subset L^1\p{\T^\N}.\eeq
In other words, $f\in H^1_{\mathrm{last}}\p{\T^\N}$ iff $\supp\widehat{f}$ lies in the positive cone of the partial order $\geq_{\mathrm{last}}$ on $\Z^{\oplus\N}$ defined by $n>_{\mathrm{last}}0$ iff $n_j>0$ for $j=\max \supp n$. Equivalently, $f\in H^1_{\mathrm{last}}\p{\T^\N}$ iff $\Delta_k f$, which is a function of first $k$ variables, is an $H^1_0\p{\T}$ function of $x_k$. It is known that for $f\in H^1_{\mathrm{last}}\p{\T^\N}$, 
\beq \|f\|_{H^1_{\mathrm{last}}\p{\T^\N}}\simeq \left\|\p{\sum_k \left|\Delta_k f\right|^2}^\frac12\right\|_{L^1\p{\T^\N}}.\eeq
In fact, more is true. A theorem due to M\"uller \cite{mulldecomp} states in particular that the Davis--Garsia decomposition can be done within the class of Hardy martingales: 
\beq \|f\|_{H^1_{\mathrm{last}}\p{\T^\N}}\simeq \inf_{\substack{f=g+h\\g,h\in H^1_{\mathrm{last}}\p{\T^\N}}} \sum_k \E\left|\Delta_k g\right| + \E\p{\sum_k \E_{k-1}\left|\Delta_k h\right|^2}^\frac12. \eeq
In other words, by the usual identification of $f$ and $\p{\Delta_k f}_{k=1}^\infty$,
\beq \label{eq:biginterpsum} H^1_{\mathrm{last}}\p{\T^\N}\sim \p{\bigoplus_{k\geq 1} L^1\p{\T^{k-1},H^1_0\p{\T}}}_{\ell^1} + L^1\p{\T^\N, \left[\p{\mathcal{F}_{k-1}}_{k=1}^\infty\right],\ell^2\p{\N,H^2_0\p{\T}}},\eeq
where in the second summand, at each $k\in \N$, the last $\T$ corresponds to $x_k$. This allows us to prove
\begin{thm}The norm of a positive sequence $\p{\lambda_n}_{n>_{\mathrm{last}}0}$ as an $H^1_{\mathrm{last}}\p{\T^\N}\to\ell^1$ multiplier is equivalent to 
\beq \sup_k \left\| \p{\sum_{n_{<k}}\lambda_{n_{<k},n_k}}_{n_k\in \Z_+}\right\|_{F} + \sup_k \p{\sum_{n_{>k}\in \Z^{[k+1,\infty)}\setminus \{0\}} \p{\sum_{n_{\leq k}\in \Z^k}\lambda_{n_{\leq k},n_{>k}}}^2}^\frac12. \eeq
\end{thm}
\begin{proof}
In order for the multiplier operator to be bounded on the interpolation sum, it has to be bounded on each of its summands. In order for $\p{\lambda_{n_{<k},n_k}}_{n_{<k}\in \Z^{k-1},n_k\in \Z_+}$ to act on a single $L^1\p{\T^{k-1},H^1_0\p{\T}}$, the inequality
\beq \sum_{n_{<k},n_k}\lambda_{n_{<k},n_k}\left|\widehat{f}\p{n_{<k},n_k}\right|\lesssim \|f\|_{L^1\p{\T^{k-1},H^1_0\p{\T}}}\eeq
has to be satisfied. By testing on the functions of the form $\varphi\otimes \psi$, where $\psi\in H^1_0\p{\T}$ and $\widehat{\varphi}\to 1$, we see that the condition
\beq \left\| \p{\sum_{n_{<k}}\lambda_{n_{<k},n_k}}_{n_k\in \Z_+}\right\|_{F}\lesssim 1\eeq
has to be satisfied. On the other hand, 
\begin{align} \sum_{n_{<k},n_k}\lambda_{n_{<k},n_k}\left|\widehat{f}\p{n_{<k},n_k}\right| \leq & \sum_{n_{<k},n_k}\lambda_{n_{<k},n_k}\int_{\T^{k-1}}\d x \left|\widehat{f\p{x,\cdot}}\p{n_k}\right|\\
\leq & \int_{\T^{k-1}}\d x \left\| \p{\sum_{n_{<k}}\lambda_{n_{<k},n_k}}_{n_k\in \Z_+}\right\|_{F} \left\|f\p{x,\cdot}\right\|_{H^1_0\p{\T}}\\
= & \left\| \p{\sum_{n_{<k}}\lambda_{n_{<k},n_k}}_{n_k\in \Z_+}\right\|_{F} \|f\|_{L^1\p{\T^{k-1},H^1_0\p{\T}}}.\end{align}
Therefore, the condition for $\lambda$ to act boundedly on the first summand of \eqref{eq:biginterpsum} is
\beq \sup_k \left\| \p{\sum_{n_{<k}}\lambda_{n_{<k},n_k}}_{n_k\in \Z_+}\right\|_{F}\lesssim 1.\eeq
For the second summand, we apply Theorem \ref{feffadap} directly to get the necessary and sufficient condition
\begin{align} 1\gtrsim &\sup_k \sum_{j\geq k} \sum_{n_{j+1}\in \Z_+}\sum_{n_{[k+1,j]}\in \Z^{[k+1,j]}}\p{\sum_{n_{[1,k]}\in \Z^{[1,k]}} \lambda_{n_{[1,k]},n_{[k+1,j]},n_{j+1}}}^2\\
=& \sup_k \sum_{n_{>k}\in \Z^{[k+1,\infty)}\setminus \{0\}} \p{\sum_{n_{\leq k}\in \Z^k}\lambda_{n_{\leq k},n_{>k}}}^2.\end{align}
\end{proof}

\end{document}